\documentclass[12pt]{amsart}
\usepackage{graphicx}
\usepackage{tikz}        
\usepackage{amssymb}
\usepackage{epstopdf}
\usepackage[mathscr]{eucal}
\usepackage{amsmath,amssymb,amscd}
\usepackage{color}
\usepackage{epsfig}
\usepackage{hyperref}
\newtheorem{theorem}{Theorem}
\newtheorem{lemma}{Lemma}

\newtheorem{definition}{Definition}
\newtheorem{corollary}{Corollary}

\newtheorem{proposition}{Proposition}
\theoremstyle{definition}

\def \mb{\mathbb}

\def \Z{\mb Z}                  
\def \N{\mb N}                
\def \R{\mb R}                 
\def \C{\mb C}                 

\newcommand {\q} {\mathbf{q}}

\def \and{\mbox{and}}
\usepackage[top=4cm, bottom=4cm, left=3.5cm, right=3.5cm]{geometry}
\title[Stability of Periodic orbits]{Stability of Periodic Orbits by Conley-Zehnder index theory}
\author{Yanxia Deng, Daniel Offin}
\email{yd17@uvic.ca, offind@mast.queensu.ca}
\address{Department of Mathematics and statistics \\University of Victoria\\ Victoria, BC, V8P 5C2 \\Canada}
\address{Department of Mathematics and statistics \\Queen's University\\ Kingston, ON K7L 3N6 \\Canada}
\date{\today}

\begin{document}
\maketitle
\begin{abstract}
We give a necessary and sufficient condition for strong stability of low dimensional Hamiltonian systems, in terms of the iterates of a closed orbit and the Conley-Zehnder index. Applications to Mathieu equation and stable harmonic oscillations for forced pendulum type equations are considered as applications of the main result. 

\end{abstract}



\section{Introduction}

The determination of stability type for periodic orbits in Hamiltonian systems is the first important goal of dynamics, after the question of existence. For linearized stability, the main tool in this quest is the theory of Liapunov multipliers and the consequent structure of the monodromy matrix for a periodic orbit \cite{MeyHallOff}. In this note we will present an alternative and equivalent formulation of linearized stability in terms of the Conley-Zehnder index theory for one degree of freedom time-dependent Hamiltonian systems. 

 The variational method using either the direct method of the calculus of variations or a saddle point method such as the mountain pass theorem for determining existence of closed orbits has experienced huge success in the preceeding period of four decades \cite{ChMo, KCC, Ferrario,mw}. There have been many attempts during the past century to tie the variational method to the determination of stability type, beginning with Poincar\'e and his description of instability for arclength minimizing closed geodesics on an orientable surface. In his book Dynamical Systems \cite{Birkhoff},  Birkhoff conjectured  that the direct method would always lead to Lyapunov instability independent of the dimension, however this was not to be the case as counterexamples show \cite{offin1}. Some extenstions of Poincar\'e's result may be found in \cite{ BolTres, deng, husun, offin1}. An earlier paper \cite{offin2}, addresses the question of stability for planar time dependent systems using  the Morse index for periodic and anti-periodic  boundary conditions to give a necessary and sufficient condition for strong stability . 

More recently, Ortega \cite{ort} has studied the classical forced pendulum equation $\ddot x + \beta \sin x = f(t)$  for T-periodic forcing $f(t)$ with mean value zero, $\int_0^Tfdt=0$. In his paper Ortega has given a simple condition on the parameters of the forced pendulum equation  to guarantee at least one strongly stable solution of the system independent of the size of the forcing. As pointed out by Ortega, the condition $\int_0^Tfdt=0$ allows a global variational approach to address  the question of existence for  periodic solutions of the pendulum system. In particular the functional 
\[ A_T(q) = \int_0^T \left ( \frac{1}{2} \dot q^2+ \beta \cos(q) +qf \right ) dt, \qquad q \in H^1_T(S^1) \]
where $H^1_T(S^1)$ denotes the Sobolev space of absolutely continuous functions on [0,T] with periodic boundary conditions $q(0)=q(T)$, has at least two geometrically distinct critical points $q_1, q_2 \in H^1_T(S^1)$ \cite{mw}. The first $q_1$ is a minimizing critical point and the second $q_2$ is a mountain pass. The minimizing solution is Lyapunov unstable by the result of Poincar\'e. The stability status of the mountain pass can be resolved, with certain restrictions, using the methods of this paper.  
 
In the sequel we will recover Ortega's result as an application of our main result, Theorem \ref{thm_dc}. Moreover we will show that this stable solution most often arises as a mountain pass critical point of an associated functional with periodic boundary conditions. 

 Of the many dynamical properties exhibited in simple low dimensional Hamiltonian systems,  stability and instability of periodic orbits have implications for the global dynamics of the system. The present work gives a necessary and sufficient condition for strong stability, in terms of the iterates of a closed orbit, which we describe in section 4. In Section 2 and 3 we review the Conley-Zehnder index and its relation to periodic Hamiltonian system and the Morse index. In section 5, applications to Mathieu equation and stable harmonic oscillations for forced pendulum type equations are considered as applications of the main result.

\section{Conley-Zehnder index for $Sp(2)$}

The Conley-Zehnder index is an important invariant of linear periodic Hamiltonian systems. It plays an important role in connection with a suitably defined Morse index of periodic solutions of non-linear systems. The index was introduced by Conley-Zehnder \cite{CZ1} for the non-degenerate case with $n\geq 2$, Long-Zehnder in \cite{LZ} for the non-degenerate case with $n=1$, and Long \cite{long1} and Viterbo \cite{viterbo} independently for the degenerate case. 

Roughly speaking, the Conley-Zehnder index is the number of half windings made by the fundamental solution of a linear Hamiltonian system in the symplectic group. For the reader's convenience, we review the Conley-Zehnder index on $Sp(2)$. Those readers already with a good knowledge of the Conley-Zehnder index may skip this section.


As usual, let $Sp(2n)$ denote the symplectic group,
$$Sp(2n)=\{A\in\mathbb{R}^{2n\times 2n}|A^TJA=J\}$$
where $$J=\left(
           \begin{array}{cc}
             0 & -I_n \\
             I_n & 0 \\
           \end{array}
         \right)$$

such that $(\mathbb{R}^{2n},\omega_0)$ with $\omega_0(v,w):=Jv\cdot w$ is the standard symplectic vector space. Here $v\cdot w$ is the usual inner product on $\mathbb{R}^{2n}$. Let us denote $$\mathcal{P}(2n):=\{\gamma\in C^0([0,T],Sp(2n))|\gamma(0)=I\}$$ the space of all the continuous symplectic paths, and $$\mathcal{P}^*(2n):=\{\gamma\in \mathcal{P}(2n)|\det(\gamma(T)-I)\neq 0\}$$ the space of all the non-degenerate symplectic paths. 

When $n=1$, the symplectic group $Sp(2)$ coincides with $SL(2, \mathbb{R})$, i.e. $2\times 2$ matrices with determinant 1. Some of the materials in this section can be found in \cite{alberto, CZ1, long, MeyHallOff} in  greater generality than presented here.

\begin{proposition}
The symplectic group $Sp(2)$ is homeomorphic to $\mathbb{R}^{2}\times S^1$, where $S^1$ is the unit circle in the complex plane.
\end{proposition}

\begin{proof}
Every invertible matrix $A$ can be written in polar form
$$A=PO.\quad P:=(AA^T)^{1/2},\quad O:=P^{-1}A$$ 
where $P$ is symmetric and positive definite, and $O$ is orthogonal. Such a decomposition is unique and it depends continuously on the matrix $A$. Now, if $A$ is symplectic, i.e. $A=J^{-1}A^{-T}J$, then
$$PO=J^{-1}(PO)^{-T}J=J^{-1}(P)^{-T}O^{-T}J=J^{-1}P^{-T}JJ^{-1}O^{-T}J$$
since $J$ is orthogonal, the matrix $J^{-1}P^{-T}J$ is still symmetric and positive definite and $J^{-1}O^{-T}J$ is orthogonal. By the uniqueness of the polar decomposition, we have
$$J^{-1}P^{-T}J=P, \quad J^{-1}O^{-T}J=O$$
which means $P$ and $O$ are symplectic. Thus $Sp(2)$ is homeomorphic to the direct product of symplectic symmetric positive definite matrices and symplectic orthogonal matrices.

Let $S(2)$ be the set of symplectic, symmetric and positive definite $2\times 2$ matrices, it is homeomorphic to $\mathbb{R}^2$. In fact, every symmetric positive definite matrix $P$ can be expressed as $P=e^S$, where $S$ is symmetric; and $P$ is symplectic if and only if $S$ satisfies
$$S^TJ+JS=0$$
Therefore, $S(2)$ is homeomorphic to the vector space
$$\{S\in\mathbb{R}^{2\times 2}: S^T=S, \,S^TJ+JS=0\}$$
The dimension of this vector space is 2.

The set of symplectic orthogonal matrices coincides with the unitary group, and in particular $O(2)\cap Sp(2)=U(1)$ which is homeomorphic to $S^1$. Hence $Sp(2)\cong\mathbb{R}^2\times S^1$.
\end{proof}

\begin{corollary}

The fundamental group of $Sp(2)$ is the free cyclic group;
$\pi_1(Sp(2))=\mathbb{Z}$.
\end{corollary}

There is a symplectically invariant and continuous function from $Sp(2)$ to $S^1$ which induces an isomorphism between their fundamental groups. We will call it the rotation function. To define this function, we need some information about the eigenvalues and the Krein signature on $Sp(2)$. 

The eigenvalues of $A\in Sp(2)$ must be of the form $\lambda, \,\frac{1}{\lambda}$, where $\lambda\in S^1\cup\mathbb{R}$. So the eigenvalues of $A$ are either conjugate pairs on the unit circle or real reciprocals. The eigenvalues $1$ and $-1$ are always double.

It is useful to introduce the Hermitian form on $\mathbb{C}^2$: $$g(v,w):=\langle Gv,w\rangle, \quad G:=-iJ=\begin{pmatrix}0&i\\-i&0\end{pmatrix}$$
where $\langle \cdot,\cdot\rangle$ is the standard Hermitian product on $\mathbb{C}^{2}$.
The complex symplectic group $Sp(2,\mathbb{C})$ consists of the g-unitary complex linear automorphisms of $\mathbb{C}^{2}$. Equivalently, $A\in Sp(2,\mathbb{C})$ if and only if $A^*GA=G$, where $A^*=\bar{A}^T$ is the conjugate transpose. A matrix belongs to $Sp(2)$ if and only if it is in $Sp(2,\mathbb{C})$ and it is real.

Assume that $A\in Sp(2)$ has eigenvalues $\lambda\neq\pm 1$ and $\bar{\lambda}$ on the unit circle and that $v$ and $\bar{v}$ are the corresponding eigenvectors. Then
$$\langle Gv, \bar{v}\rangle=\langle A^*GAv, \bar{v}\rangle=\langle GAv, A\bar{v}\rangle=\lambda^2\langle Gv, \bar{v}\rangle$$

Since $\lambda\neq\pm1$, $\langle Gv, \bar{v}\rangle=0$. So $\{v, \bar{v}\}$ is a $G$-orthogonal basis of $\mathbb{C}^2$. Hence $\langle Gv, {v}\rangle$ and $\langle G\bar{v}, \bar{v}\rangle$ are real and non-zero. This justifies the following definition.

\begin{definition}[Krein sign]
If $\lambda\in S^1\setminus \{\pm 1\}$ is an eigenvalue of $A\in Sp(2)$ and $v$ is the corresponding eigenvector, the Krein sign of $\lambda$ is the sign of $\langle Gv, {v}\rangle$. If the Krein sign is positive we say $\lambda$ is Krein-positive.
\end{definition}

Since the signature of $G$ is $(1,1)$, if $\lambda\in S^1\setminus \{\pm 1\}$ is Krein-positive, then $\bar{\lambda}$ is Krein-negative. For the double eigenvalues $\pm1$, we will consider them as a pair of eigenvalues, one of which is Krein-positive, the other is Krein-negative.

Now we are ready to define the rotation function.

Let $\lambda, \lambda^{-1}$ be the eigenvalues of $A\in Sp(2)$ in the form where $|\lambda|<1$ or $|\lambda|=1$ and $\lambda$ is Krein-positive. The rotation function $$\rho: Sp(2)\to S^1$$ is defined as
$$\rho(A):=\frac{\lambda}{|\lambda|}$$

\begin{theorem}[cf. \cite{alberto}, \cite{long}] \label{rotation}
The rotation function $\rho$ defined above has the following properties:
\begin{itemize}
\item[(1)](Continuity) The function $\rho$ is continuous.
\item[(2)](Symplectic invariance) $\rho(MAM^{-1})=\rho(A)$ for any $A,M$ in $Sp(2)$.
\item[(3)](Normalization) $\rho(A)=\pm 1$ if $A$ does not have eigenvalues on the unit circle.
\item[(4)](Value on the unitary group) If $A\in Sp(2)\cap O(2)$, and $U$ is its corresponding unitary matrix in $U(1)$, then $\rho(A)=\det(U)$.
\item[(5)](Homotopy) The map $\rho$ induces an isomorphism of fundamental groups.
\end{itemize}
\end{theorem}

\begin{proof}
Continuity, symplectic invariance and normalization is straightforward from the definition of $\rho$. Let us show (4) and (5). If $A$ is in $Sp(2)\cap O(2)=SO(2)=U(1)$, then $A$ has the form
$$R(\theta)=\begin{pmatrix}\cos\theta&-\sin\theta\\\sin\theta&\cos\theta\end{pmatrix}$$
for some $\theta\in[0, 2\pi]$. The eigenvalues of $A$ are $e^{\pm i\theta}$ and an eigenvector for $e^{i\theta}$ is $$v=\begin{pmatrix}\cos\theta+i\sin\theta\\\sin\theta-i\cos\theta\end{pmatrix}$$
Calculating $$\langle Gv, {v}\rangle=2$$
So $e^{i\theta}$ is Krein-positive and $\rho(A)=e^{i\theta}$. Therefore, $\rho$ restricting on $Sp(2)\cap O(2)$ coincides with the map
$$\mu: SO(2)\to U(1)$$
$$R(\theta)\mapsto e^{i\theta}$$
Thus $\rho$ is homotopic to $\mu$, and $\rho$ induces an isomorphism between the fundamental groups of $Sp(2)$ and $S^1$.
\end{proof}

We need a little more information about the topology of $Sp(2)$ before we can define the Conley-Zehnder index. Namely, $Sp(2)$ can be divided into three components:
$$Sp(2)^\pm=\{A\in Sp(2)|\pm \det(A-I)<0\}$$
$$Sp(2)^0=\{A\in Sp(2)|\det(A-I)=0\}$$
Let $Sp(2)^*:=Sp(2)^+\cup Sp(2)^-$, and we fix two representatives $M^{\pm}$ in $Sp(2)^{\pm}$ with 
$$M^+=\begin{pmatrix}2&0\\0&\frac{1}{2}\end{pmatrix}\quad\text{and}\quad M^-=\begin{pmatrix}-2&0\\0&-\frac{1}{2}\end{pmatrix}$$

Easy calculation shows that $\rho(M^+)=1$ and $\rho(M^-)=-1$.
\begin{theorem}[Long \cite{long1999}]\label{con}
The open sets $Sp(2)^+$ and $Sp(2)^-$ are path connected and simply connected in $Sp(2)$.
\end{theorem}
\begin{proof}
Elementary calculation shows that $A\in Sp(2)^+$ if and only if $A$ has a pair of positive real eigenvalues not equal to 1, and $A\in Sp(2)^-$ if and only if $A$ has a pair of negative real eigenvalues or a pair of complex conjugate eigenvalues on the unit circle not equal to 1. So the connectedness of $Sp(2)^{\pm}$ is easy to see. Let us show that $Sp(2)^{\pm}$ are simply conncected in $Sp(2)$. 

Let $\gamma: S^1\to Sp(2)^+$  be a continuous loop. Since the rotation function $\rho: Sp(2)^+\to S^1$ is the constant function $\rho\equiv1$, and $\rho$ induces an isomorphism of the fundamental groups, so any loop in $Sp(2)^+$ is contractible.

For $Sp(2)^-$, assume $A\in Sp(2)^-$ has eigenvalues $\lambda, \lambda^{-1}$ with $|\lambda|<1$ or $|\lambda|=1$ and $\lambda$ is Krein-positive. Then there is a well-defined number $\theta(A)$ so that $$\rho(A)=\frac{\lambda}{|\lambda|}=e^{i\theta(A)}$$ with $0<\theta(A)<2\pi$. Let $\gamma: S^1\to Sp(2)^-$  be a continuous loop, $\gamma$ is contractible if and only if $\rho\circ\gamma: S^1\to S^1$ is contractible. 
Since $$\rho(\gamma(t))=e^{i\theta(t)}$$
and $\theta(t): S^1\to (0, 2\pi)$, $\rho\circ\gamma$ is contractible.
\end{proof}

Now we are ready to define the Conley-Zehnder index. Given a path $\gamma\in\mathcal{P}^*(2)$, we can extend it to a continuous path $\tilde{\gamma}:[0,2T]\to Sp(2)$ such that 

\begin{itemize}
\item $\tilde{\gamma}|_{[0,T]}=\gamma$,
\item $\tilde{\gamma}(t)\in Sp(2)^\pm$ for $t\geq T$,
\item $\tilde{\gamma}(2T)\in\{M^+,M^-\}$.
\end{itemize}

that is, $\tilde{\gamma}$ coincides with $\gamma$ on the interval $[0,T]$, $\tilde{\gamma}(t)$ remains in $Sp(2)^\pm$ for all $t\geqslant T$, and $\tilde{\gamma}(2T)$ ends up at the  corresponding representative of the component. Note that $\rho(M^\pm)\in\{\pm 1\}$, thus $(\rho\circ\tilde{\gamma})^2:[0,2T]\to S^1$ is a loop. Here the square is understood as $S^1\to S^1: z\mapsto z^2$ in the complex plane.

\begin{definition}[Conley-Zehnder index]
Given the notation as above, the Conley-Zehnder index of a path $\gamma\in\mathcal{P}^*(2)$ is the degree of the map $(\rho\circ\tilde{\gamma})^2:[0,2T]\to S^1$. Denote the Conley-Zehnder index of $\gamma$ as $i_1(\gamma)$: $$i_1(\gamma)=\deg (\rho\circ\tilde{\gamma})^2$$
\end{definition}
Note that since the open sets $Sp(2)^+$ and $Sp(2)^-$ are connected and simply connected in $Sp(2)$, the Conley-Zehnder index of $\gamma$ does not depend on the choice of the extension $\tilde{\gamma}$.
Moreover, from the definition, we have the following proposition.
\begin{proposition}[Lemma 5.2.6, \cite{long}]
\label{evenodd}
Given the notation as above,
\begin{itemize}
\item [$\circ$]$i_1(\gamma)$ is even if and only if $\tilde{\gamma}(2T)=M^+$, \text{i.e.}\,$\gamma(T)\in Sp(2)^+$,
\item [$\circ$]$i_1(\gamma)$ is odd if and only if $\tilde{\gamma}(2T)=M^-$, \text{i.e.}\,$\gamma(T)\in Sp(2)^-$.
\end{itemize}
\end{proposition}

\section{Conley-Zehnder index for periodic Hamiltonian system}

For the definition of Conley-Zehnder index for higher dimensional symplectic paths in $Sp(2n)$, we refer to \cite{alberto}, \cite{CZ1}, \cite{long}.
Now that we have Conley-Zehnder index for a symplectic path, we may wonder how do symplectic paths arise from dynamical system. It turns out that the fundamental solution of a periodic linear Hamiltonian system is a symplectic path. More precisely, a linear $T$-periodic Hamiltonian system in $\mathbb{R}^{2n}$ has the form:
$$\dot{Z}(t)=JS(t)Z(t)$$
where $S(t)$ is a $T$-periodic path of symmetric matrices. Let $\gamma(t)$ be the fundamental solution of the equation, then $\gamma(t)$ is symplectic for each $t$. 
More precisely, $\gamma(t)$ satisfies
$$\begin{cases}
\begin{aligned}
\dot{\gamma}(t)&=JS(t)\gamma(t)\\
\gamma(0)&=I
\end{aligned}
\end{cases}$$
so $$\frac{d}{dt}[\gamma(t)^TJ\gamma(t)]=\gamma(t)^TS(t)^TJ^TJ\gamma(t)+\gamma(t)^TJJS(t)\gamma(t)=0$$
and $\gamma(0)$ is symplectic, thus $\gamma(t)$ is symplectic for each $t$.

A periodic Hamiltonian system is called non-degenerate if 1 is not a Floquet multiplier. i.e. $\det(\gamma(T)-I)\neq 0$. We define the Conley-Zehnder index of a non-degenerate linear periodic Hamiltonian system to be the Conley-Zehnder index of its fundamental solution, i.e. $i_1(\gamma)$.

For a $T$-periodic solution $x(t)$ of a nonlinear $T$-periodic Hamiltonian system
$$\dot{x}(t)=JdH(x,t)$$
By Floquet theory, we can consider the linearized equation along this periodic solution:
$$\dot{Z}(t)=Jd^2H(x(t),t)Z(t)$$
then the linear stability of $x(t)$ can be obtained by the fundamental solution of the linearized equation.
Denote by $\gamma_x(t)$ the corresponding fundamental solution, the Conley-Zehnder index of $x(t)$ is then defined to be $i_1(\gamma_x)$. We say $x(t)$ is elliptic (hence linearly stable) if eigenvalues of $\gamma_x(T)$ are in $S^1\setminus \{\pm1\}$, and $x(t)$ is hyperbolic (hence linearly unstable) if eigenvalues of $\gamma_x(T)$ are in $\mathbb{R}\setminus \{\pm1\}$. 


Another application of Conley-Zehnder index is the Lagrangian system. For $T>0$, suppose $x$ is a critical point of the functional
$$F(x)=\int_0^TL(t, x, \dot{x})dt,\quad \forall x\in W$$
where $W=W^{1,2}(\mathbb{R}/T\mathbb{Z},\mathbb{R}^n)$ is a Sobolev space with the usual inner product
\[\langle x, y\rangle=\int_0^T(x\cdot y+\dot{x}\cdot \dot{y})dt\]
The Lagrangian $L\in C^2((\mathbb{R}/T\mathbb{Z})\times\mathbb{R}^{2n},\mathbb{R})$ and satisfies the Legendrian convexity condition $L_{pp}(t,x,p)>0$. It is known that $x$ will satisfy the Euler-Lagrangian equation:
$$\frac{d}{dt}L_p(t,x,\dot{x})-L_x(t,x,\dot{x})=0,$$
$$x(0)=x(T), \quad \dot{x}(0)=\dot{x}(T)$$
For such an extremal loop, define 
\begin{center}
$\begin{cases}
\begin{aligned}
P(t)&=L_{pp}(t, x(t),\dot{x}(t))\\
Q(t)&=L_{xp}(t, x(t),\dot{x}(t))\\
R(t)&=L_{xx}(t, x(t),\dot{x}(t))\\
\end{aligned}
\end{cases}$
\end{center}

Then the second variation of $F$ at $x$ is a symmetric bilinear form on $W$:
$$D^2F(x)\langle\alpha, \beta\rangle=\int_0^T\langle P\dot{\alpha}+Q\alpha,\dot{\beta}\rangle+\langle Q^T\dot{\alpha}, \beta\rangle+\langle R\alpha, \beta\rangle dt$$
where $\langle\cdot, \cdot\rangle$ is the inner product on $W$ and $\alpha, \beta\in W$. The \textbf{Morse index} of the extremal loop $x$, denoted by $m^-(x)$, is the total multiplicity of the negative eigenvalues of $D^2F(x)$. The Morse index is a finite integer when $L$ satisfies the Legendre convex condition (cf. \cite{duis}).

Moreover, the Euler-Lagrange equation corresponds to a Hamiltonian equation. The linearized Hamiltonian equation at $(x, D_pL(x, \dot{x}, t))$ is given by $$\dot{Z}(t)=JB(t)Z(t),$$
where $$B=\begin{pmatrix}P^{-1}(t)&-P^{-1}(t)Q(t)\\
-Q(t)^TP^{-1}(t)&Q(t)^TP^{-1}(t)Q(t)-R(t)
\end{pmatrix}$$
Let the fundamental solution of this Hamiltonian system be $\gamma_x$, then 

\begin{theorem}[C.Viterbo \cite{viterbo}, Y.Long and T.An \cite{longan}]
\label{czmorse}
Under the above conditions, $$m^-(x)=i_1(\gamma_x).$$
\end{theorem}
That is, the Morse index of an extremal loop is equal to its Conley-Zehnder index. For more information on these topics, we refer to \cite{long}.


\section{Iteration formula of Conley-Zehnder index}
The iteration theory of Morse type indices is very important in the study of periodic solutions. We would like to introduce an interesting iteration formula of the Conley-Zehnder index. First we need a generalization of the Conley-Zehnder index, i.e. the $\omega$-index, which was first introduced by Long \cite{long99} for every $\omega\in \C$ with modulus equal to 1. 

Given $A\in Sp(2)$, let $D(\omega)=\bar{\omega}\det(A-\omega I)$, then one can show that $D(\omega)$ is real-valued. We can divide $Sp(2)$ into three components for each $\omega\in S^1$:
$$Sp(2)_{\omega}^{\pm}=\{A\in Sp(2)|\pm D(\omega)<0\}$$
$$Sp(2)_{\omega}^{0}=\{A\in Sp(2)| D(\omega)=0\}$$

\begin{theorem}[Long \cite{long1999}]
The open sets $Sp(2)_{\omega}^{\pm}$ are connected and simply conncected in $Sp(2)$.
\end{theorem}
 The proof is similar to Theorem \ref{con}. Moreover, the two representatives $M^{\pm}$, where
 $$M^+=\begin{pmatrix}2&0\\0&\frac{1}{2}\end{pmatrix}, \quad M^-=\begin{pmatrix}-2&0\\0&-\frac{1}{2}\end{pmatrix},$$
 work for $Sp(2)_{\omega}^{\pm}$, and $M^+\in Sp(2)_{\omega}^{+}$, $M^-\in Sp(2)_{\omega}^{-}$.
 
 When $\omega=1$, these components coincide with $Sp(2)^{\pm}, Sp(2)^0$ as in the previous section. The Conley-Zehnder index $i_1(\gamma)$ for a symplectic path $\gamma$ can be viewed as detecting the change of eigenvalues of $\gamma$ at $1$. Similarly, we can detecting the eigenvalues of $\gamma$ at other points of the unit circle, say at $\omega\in S^1$.

Let $$\mathcal{P}_{\omega}^*(2):=\{\gamma\in \mathcal{P}(2)|\det(\gamma(T)-\omega I)\neq 0\}$$ be the set of $\omega-$nondegenerate symplectic paths in $Sp(2)$. If $\gamma$ is an $\omega-$nondegenerate symplectic path, then we can extend it to a continuous path $\tilde{\gamma}:[0,2T]\to Sp(2)$ such that 

\begin{itemize}
\item $\tilde{\gamma}|_{[0,T]}=\gamma$,
\item $\tilde{\gamma}(t)\in Sp(2)_{\omega}^\pm$ for $t\geq T$,
\item $\tilde{\gamma}(2T)\in\{M^+,M^-\}$.
\end{itemize}

that is, $\tilde{\gamma}$ coincides with $\gamma$ on the interval $[0,T]$, $\tilde{\gamma}(t)$ remains in $Sp(2)_{\omega}^\pm$ for all $t\geqslant T$ , whichever component $\gamma(T)$ lies in, and $\tilde{\gamma}(2T)$ ends up at the  corresponding representative of the component. Note that $\rho(M^\pm)\in\{\pm 1\}$, thus $(\rho\circ\tilde{\gamma})^2:[0,2T]\to S^1$ is a loop. 

\begin{definition}[$\omega$-index, \cite{long99}]
\label{omega_index}
Given the notation as above, the $\omega$-index of a path $\gamma\in\mathcal{P}_{\omega}^*(2)$ is the degree of the map $(\rho\circ\tilde{\gamma})^2:[0,2T]\to S^1$. Denote the $\omega$-index of $\gamma$ as $i_\omega(\gamma)$: $$i_\omega(\gamma)=\deg (\rho\circ\tilde{\gamma})^2$$
\end{definition}
Note that since the open sets $Sp(2)_{\omega}^+$ and $Sp(2)_{\omega}^-$ are connected and simply connected in $Sp(2)$, the $\omega$-index of $\gamma$ does not depend on the choice of the extension $\tilde{\gamma}$.

For $T>0$, let $x: \mathbb{R}/T\Z\to \R^2$ be a $T$-periodic solution of the time-$T$ periodic Hamiltonian system
$$\dot{x}=JdH(x,t)$$ 
We define the $m$-th iteration of $x$ by
$$x^m(t)=x(t-jT),~~\forall jT\leq t\leq (j+1)T, j=0,1,\cdots,m-1$$ 
Then $x^m$ is an $mT$-periodic solution of the Hamiltonian system. Geometrically it is the same as $x$. Let $$\dot{Z}(t)=Jd^2H(x(t),t)Z(t)$$
be the linearized Hamiltonian system at $x$, and let $\gamma(t)$ be the corresponding fundamental solution. Define the $m$-th iteration of $\gamma$ by
$$\gamma^m(t)=\gamma(t-jT)\gamma(T)^j, \quad\forall\,\, jT\leq t\leq (j+1)T, \quad j=0,1, \cdots, m-1$$
The $m$-th iteration of any symplectic path will be defined in this way.
By the uniqueness of initial value problem of linear ODE, we know $\gamma^m$ is the fundamental solution corresponding to $x^m$.

\begin{theorem}[Bott-type iteration formula, \cite{long99}] \label{iteration}
For any $\gamma\in \mathcal{P}(2)$, $z\in S^1$ and $m\in\N$,
$$i_z(\gamma^m)=\sum\limits_{\omega^m=z}i_\omega(\gamma)$$
\end{theorem}

In particular, when $z=1$, we have $$i_1(\gamma^m)=\sum\limits_{\omega^m=1}i_\omega(\gamma)$$
that is, the Conley-Zehnder index of the $m$-th iterate, $\gamma^m$, is equal to the sum of the $\omega$-index of $\gamma$, where $\omega$ are the $m$-th roots of unity.

Now we can state our main theorem. A non-degenerate 2-dimensional periodic solution is elliptic if and only if the Conley-Zehnder index of its double cover is odd.

\begin{theorem}[Stability via double cover]
\label{thm_dc}
Let $x:\R/T\Z\to \R^2$ be a $T$-periodic solution of a periodic Hamiltonian system. Suppose $x$ and $x^2$ are non-degenerate, then 
$x$ is elliptic if and only if $i_1(x^2)$ is odd, equivalently, $x$ is hyperbolic if and only if $i_1(x^2)$ is even.
\end{theorem}
\begin{proof}
Let $\gamma$ be the fundamental solution of the linearized Hamiltonian system at $x$, then $\gamma^2$ is the corresponding fundamental solution for $x^2$, and we have $\gamma^2(2T)=\gamma(T)^2$. Let $\lambda, \lambda^{-1}$ be the eigenvalues of $\gamma(T)$, then $\lambda^2, \lambda^{-2}$ are the eigenvalues of $\gamma^2(2T)$.

If $i_1(x^2)=i_1(\gamma^2)$ is odd, by Proposition \ref{evenodd} the monodromy matrix $\gamma^2(2T)$ is in $Sp(2)^-$. So the eigenvalues of $\gamma^2(2T)$ must be negative real reciprocals or a pair of complex conjugate numbers on the unit circle not equal to 1. Since the eigenvalues of $\gamma^2(2T)$ must be of the form $\{\lambda^2, \lambda^{-2}\}$, they cannot be negative real reciprocals, so $\lambda^2, \lambda^{-2}$ must be on the unit circle. Hence $\lambda, \lambda^{-1}$ are on the unit circle, and $x$ is elliptic.

On the other hand, if $x$ is elliptic, then $\lambda, \lambda^{-1}$ are on the unit circle, so are $\lambda^2, \lambda^{-2}$. Therefore $\gamma^2(2T)$ is in $Sp(2)^-$, and $i_1(x^2)=i_1(\gamma^2)$ is odd.
\end{proof}

\section{applications}
In this section we will discuss various applications of the results given in earlier sections. We will use the notation $i_2(x)=i_1(x^2)$ to denote the Conley-Zehnder index of the double cover of $x$ in our applications. 

\subsection{Mathieu equation}
We begin with a famous example, the Mathieu equation, which we will use as an illustrative example
\[ \ddot q + \left ( \omega^2 + \epsilon \cos ( 2t ) \right ) q = 0  \]
Here the solutions depend on parameters $\omega, \epsilon$ and the system is Hamiltonian with $\pi$ periodic Hamiltonian  \[H(q,p, \omega, \epsilon, t) = \frac{1}{2} p^2 + \frac{1}{2} \left ( \omega^2 +  \epsilon \cos ( 2t ) \right ) q^2. \]
We remark here that what we will describe here for the Mathieu equation has similar formulations  for any time periodic system of Hamiltonian type with one degree of freedom (also called one and one half degrees of freedom). The fundamental matrix solution $A= A(t,\omega, \epsilon) \in Sp(2)$ undergoes changes in stability type, and also in the values of the Conley-Zehnder indices $i_1$ and $i_2$ as the parameters $\omega, \epsilon$ change.  We will explain how the classical results on stability-instability for this system can be recast in terms of the Conley-Zehnder indices. First of all, the monodromy matrix $A(\pi, \omega, \epsilon)$ determines the stability of the system, and in particular the eigenvalues  $\lambda_1, \lambda_2$ of A,  which are reciprocal in the complex plane. The following well known stability diagram (Figure \ref{fig:stability}), describes the regions of stability and instability in the $\omega^2,\epsilon$ parameter plane. The main feature for our purposes, will be the stability curves which cross on the $\omega^2$ axis, and such crossings, when $\epsilon=0$,  correspond to the classical eigenvalues $\mu_n = n^2$ of the corresponding Sturm-Liouville self adjoint operator on $L^2[0, \pi]$, for $\pi$ periodic and $\pi$ anti-periodic boundary conditions. Classically, such curves are determined by power series methods (in powers of $\epsilon$) for $\pi$  periodic or $\pi$ anti-periodic solutions. The transitions between stability and instability for the Mathieu system, are indicated by the stability curves. Crossing one or several of these curves will change the stability of the system. The stable domains are the open regions including intervals on the $\omega^2$ axis. The domains of instability bifurcate from the eigenvalues $\mu_n = n^2$ on the $\omega^2$ axis.

\begin{figure}[h] 
\begin{center}
\includegraphics[scale=.7]{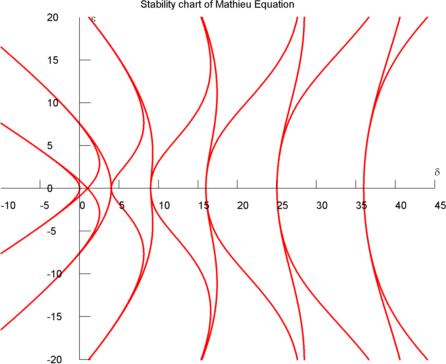}
\caption{Bifurcation of stability and instability domains in $\omega^2=\delta,\epsilon$ plane.}\label{fig:stability}
\end{center}
\end{figure}

There is a direct way to see how the mapping $\rho: Sp(2) \rightarrow S^1$ can be visualized. We introduce a mapping $Sp(2) \times S^1 \rightarrow S^1$ using angular coordinate $\theta$ in the projective space $P^1(\mathbb R)$ which is just the angle $\theta$ from polar coordinates in the $p, q$ plane. The  period mapping $A(\pi,\omega,\epsilon)$  for the Hamiltonian $H(q,p, \omega, \epsilon, t)$ gives an element of $Sp(2)$  and  the projective point $(p(0),q(0))$  an element of $S^1=P^1(\mathbb R)$. The mapping that we consider is then 
\[ Sp(2) \times S^1 \rightarrow S^1, \qquad  ( A(\pi,\omega,\epsilon), \theta(0) ) \mapsto \theta(\pi) \] The angle $\theta$ can easily be seen to satisfy the relation
\[ \dot {\theta}  = \frac{2H(q(t),p(t), t, \omega, \epsilon)}{p^2 + q^2} \] 
where $(p(t), q(t))$ is a solution to the Mathieu equation with initial condition $(p(0),q(0))$. Now we will focus on the change in the angle $\theta$ over one period of the monodromy mapping. 
\begin{equation} \label{delta} \Delta \theta ( \omega, \epsilon, p(0), q(0) ) = 2 \int_0^\pi \frac{H(q(t),p(t), t, \omega, \epsilon)}{p^2 + q^2}dt  
\end{equation} 
We will be concerned with the following condition
\begin{equation} \label{resonance} \Delta \theta = k \pi , \qquad k \in \mathbb{Z} 
\end{equation}
which will allow us to locate the stability transition curves in Figure \ref{fig:stability}.
\begin{lemma}
\label{rotation} 
The condition (\ref{resonance}) is equivalent to the nonexistence of eigenvalues of the monodromy matrix $A(\pi,\omega,\epsilon)$ of elliptic type $e^{i\alpha}, 0< \alpha < \pi$. The rotation function $\rho$  described in Theorem \ref{rotation} takes values $\pm 1$ on the locus of points in the $\omega, \epsilon$ plane where (\ref{resonance}) holds. Moreover, if $\rho = \pm 1$ then condition (\ref{resonance}) holds. 
\end{lemma}
\begin{proof} The condition (\ref{resonance}) is equivalent to the existence of an invariant line in the $p, q$ plane for the monodromy matrix  $A(\pi,\omega,\epsilon)$. This happens if and only if the eigenvalues of the monodromy matrix are real reciprocal non zero numbers. 
\end{proof}
The eigenvalues of the monodromy matrix vary continuously with the parameters $\omega, \epsilon$. Therefore we have the following 
\begin{lemma}
\label{boundary}
The boundary of the regions in the $\omega^2,\epsilon$ plane where condition (\ref{resonance}) holds,  correspond to the locus of points where the eigenvalues of the monodromy matrix $A(\pi,\omega,\epsilon)$ are $\pm 1$. 
\end{lemma}
To see how this can be used to actually determine the stability transition curves emanating from the eigenvalues $\mu_n$ on the $\omega^2$ axis, we consider the case where $\epsilon=0$ in the Mathieu equation first. Namely choosing initial conditions $p(0)= -n\sin(0), q(0) = \cos(0)$ we have 
\begin{equation} \label{resonance2}   \Delta \theta ( n, 0, p(0), q(0) ) = n\pi \end{equation}
We  extend this condition into the region where $\epsilon \neq 0$ using the implicit function theorem
\begin{proposition}
There exists $\epsilon_0= \epsilon_0(n) > 0$ such that the condition (\ref{resonance2}) holds for  every value of the initial condition $(p(0),q(0))$ and for $\omega=\omega(\epsilon,p(0),q(0)),  | \epsilon | < \epsilon_0$. The function $\omega(\epsilon,p(0),q(0))$ is analytic on its domain and  $\omega(0,p(0),q(0))= n$. 
\end{proposition}

Now we can use these results to see how the Conley-Zehnder indices change as we move in the $\omega^2,\epsilon$ plane. First we compute the Conley-Zehnder index on the $\omega^2$-axis, i.e. when $\epsilon=0$. We will assume $\omega\geq 0$. The Hamiltonian of the Mathieu equation for $\epsilon=0$ is
\[H(q,p, \omega, 0, t) = \frac{1}{2} p^2 + \frac{1}{2} \omega^2 q^2.\]
The Hamiltonian equation is \[\begin{pmatrix}\dot{p}\\\dot{q}\end{pmatrix}=\begin{pmatrix}0&-\omega^2\\1&0\end{pmatrix}\begin{pmatrix}p\\q\end{pmatrix}\]

The fundamental solution for this Hamiltonian equation is symplcetically conjugate to \[\gamma(t)=R(\omega t)=\begin{pmatrix}\cos\omega t&-\sin\omega t\\\sin\omega t&\cos\omega t\end{pmatrix}\]
Direct computations will give us the indices $i_1, i_2$ for $\gamma$,
\[i_1(\gamma)=2n+1, \quad\text{if} \,\,2n<\omega<2n+2\]
\[i_2(\gamma)=2n+1, \quad\text{if} \,\,n<\omega<n+1\]
Thus from Theorem \ref{thm_dc} or from the fundamental solution $\gamma(t)$ itself, the Mathieu equation is elliptic and linearly stable on each open intervals $\omega^2\in (n^2, (n+1)^2)$ when $ \epsilon=0$.

For the $\omega^2, \epsilon$ plane, we can make the general observation that the index $i_1$ changes by either 1, or 2 whenever the eigenvalues of the monodromy matrix are repeated at +1, and the index $i_{-1}$ changes by either 1, or 2 whenever the eigenvalues of the monodromy matrix are repeated at -1 (cf. Definition \ref{omega_index}). The specific change in these indices depends on the crossing of the eigenvalues of the monodromy matrix at $\pm1$ as the parameters $\omega, \epsilon$ vary. In particular, when the eigenvalues $\lambda_1, \lambda_2$ of $A(\pi, \omega, \epsilon)$ repeat at $\pm1$ and change from the unit circle to the real line, and vise versa, then $i_{\pm1}$ changes by $1$; when $\lambda_1, \lambda_2$ repeat at $\pm1$ and remain on the unit circle before and after the collision, then $i_{\pm 1}$ changes by $2$.

\begin{figure}
	\centering
	\begin{tikzpicture}

	\draw [->] (-0.3, -2)--(-0.3, 2)node at (-0.3,1.9) [left] {$\epsilon$};
	\draw [->](-0.5, 0)--(9.6, 0) node  [below] {$\omega^2$};
	\draw [dotted](-0.4, 0.9)--(9.6, 0.9) ;
         \draw [thick] (0.1,-1) to [out=45, in=-95] (0.8,1);
         \draw [ thick] (0.19,1) to [out=-45, in=-85] (0.8,-1);
          \draw [thick] (1.6,-1) to [out=45, in=90] (1.9,0) to [out=90,in=-45] (1.6, 1);
          \draw [thick] (2.2,-1) to [out=135, in=90] (1.9,0) to [out=90,in=-135] (2.2, 1);
          \draw [thick] (4,-1) to [out=45, in=90] (4.3,0) to [out=90,in=-45] (4, 1);
          \draw [thick] (4.6,-1) to [out=135, in=90] (4.3,0) to [out=90,in=-135] (4.6, 1);
          \draw [thick] (8,-1) to [out=45, in=90] (8.3,0) to [out=90,in=-45] (8, 1);
          \draw [thick] (8.6,-1) to [out=135, in=90] (8.3,0) to [out=90,in=-135] (8.6, 1);
\node  at (0.8, 0) [below]  {\tiny 1};		
\node at (2.2, 0) [below]  {\tiny4};		
\node at (4.6, 0) [below]  {\tiny9};		
\node at (8.6, 0) [below]  {\tiny16};		

\node [red] at (-0.6, 0.95) [above]  {$i_2$};
\node [blue] at (-0.6, 0.95) [below]  {$i_1$};
\node [red] at (0, 0.95) [above]  {\tiny1};
\node [blue] at (0, 0.95) [below]  {\tiny1};	
\node [red] at (0.5, 0.95) [above]  {\tiny2};
\node [blue] at (0.5, 0.95) [below]  {\tiny1};	
\node [red] at (1.1, 0.95) [above]  {\tiny3};
\node [blue] at (1.1, 0.95) [below]  {\tiny1};
\node [red] at (1.9, 0.95) [above]  {\tiny4};
\node [blue] at (1.9, 0.95) [below]  {\tiny2};	
\node [red] at (3.1, 0.95) [above]  {\tiny5};
\node [blue] at (3.1, 0.95) [below]  {\tiny3};
\node [red] at (4.3, 0.95) [above]  {\tiny6};
\node [blue] at (4.3, 0.95) [below]  {\tiny3};
\node [red] at (6.3, 0.95) [above]  {\tiny7};
\node [blue] at (6.3, 0.95) [below]  {\tiny3};
\node [red] at (8.3, 0.95) [above]  {\tiny8};
\node [blue] at (8.3, 0.95) [below]  {\tiny4};
\node [red] at (9.2, 0.95) [above]  {\tiny9};
\node [blue] at (9.2, 0.95) [below]  {\tiny5};

	\end{tikzpicture}
	\caption{ The indices $i_1, i_2$ for Mathieu equation in $\omega^2, \epsilon$ plane. Odd values of $i_2$ are stability regions and even values of $i_2$ are instability regions.} \label{fig:mathieu}
\end{figure}
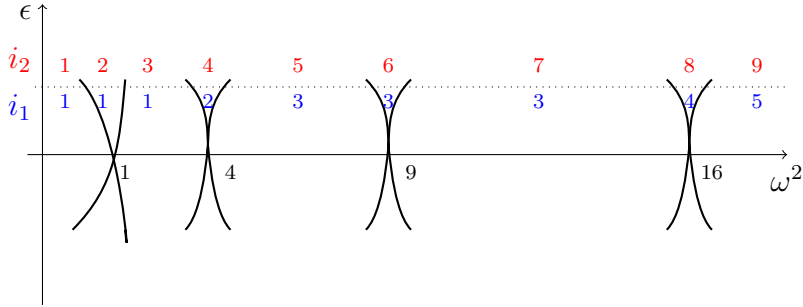

According to Lemma \ref {rotation} and Lemma \ref{boundary}, the eigenvalues of $A(\pi, \omega, \epsilon)$ when repeat at $\pm1$ always change from the unit circle to the real line or vise versa, when we cross the the boundary of the regions in the $\omega^2,\epsilon$ plane where condition (\ref{resonance}) holds. Therefore $i_{1}$ changes by $1$ each time when crossing such boundaries emanating from $\mu_{2n}=4n^2$ and $i_{-1}$ changes by $1$ each time when crossing such boundaries emanating from $\mu_{2n+1}=(2n+1)^2$. By Theorem \ref{iteration}, $i_2=i_1+i_{-1}$, thus $i_{2}$ changes by $1$ each time when crossing such boundaries. We will call those boundaries the stability transition curves. 

To make this clearer, we consider a crossection of the stability diagram with fixed value of $\epsilon >0$ (but not too large so that the intersection of the transition curve coming from $\omega=1$ and the $\epsilon$ axis lies above the value $\epsilon$) see Figure \ref{fig:mathieu}. The index $i_1$ starts with value +1 in the first stability domain (which contains the interval (0,1) on the $\omega^2$ axis). The first stability domain corresponds to $i_2 = 1$, and as we cross the first stability transition curve (emanating from the eigenvalue $\mu_1 = 1$) and enter the region where condition (\ref{resonance}) holds for $k=1$, the index $i_2$ jumps by +1. Crossing the domain of instability keeps the indices at the values $i_1=1, i_2=2$, and the next jump to $i_2=3$ occurs as the value of $\omega$ crosses the next transition curve (also emanating from $\mu_1 =1$). The next domain of stability (containing the open interval (1,4) on the $\omega^2$ axis)  is characterized by the values $i_1=1, i_2=3$ and this persists until the crossection encounters the next transition curve (emanating from the eigenvalue $\mu_2=4$) , where $i_1$ jumps by +1 to $i_1=2$. The index $i_2$ also jumps by +1 to $i_2=4$ at this crossing, and persists throughout this domain of instability (bifurcating from the eigenvalue $\mu_2 =4$). This pattern persists with successive stability domains characterized by the index $i_1 = 2k+1$ and $i_2 = i_1+l$ or $i_2=i_1+l+2$ where $l$ is an even interger, and instability domains characterized by the index $i_1=2k$ and $i_2=4k$, see figure \ref{fig:mathieu}. The domains of stability are characterized by odd values of $i_2$, and instability domains by even values of $i_2$, as indicated by Theorem \ref{thm_dc}.

\subsection{The forced pendulum equation}
We recall the basic information on the forced pendulum equation, with periodic mean forcing zero. The governing equation is 
\begin{equation} \label{fp}  \ddot x + \beta \sin(x) = f(t), \qquad f(t+T) = f(t), \qquad \int_0^T fdt = 0 \end{equation}
As mentioned in the introduction, this is the Euler equation for the functional 

\[ A_T(q) = \int_0^T \left ( \frac{1}{2} \dot q^2+ \beta \cos(q) +qf \right ) dt, \qquad q \in H^1_T(S^1) \]
where $H^1_T(S^1)$ denotes the Sobolev space of absolutely continuous functions on [0,T], with $L^2[0,T]$ derivatives, and   periodic boundary conditions $q(0)=q(T)$. This functional  has at least two geometrically distinct critical points $q_1, q_2 \in H^1_T(S^1)$ \cite{mw}. The first $q_1$ is a globally minimizing critical point. Since $A_T(q_1 + 2\pi) = A_T(q_1)$, it follows that $q_1 + 2\pi$ is also globally minimizing and it can be shown that there is a mountain pass critical value and corresponding critical point $q_2$
\[ c = \min_{g \in G} \max_{0 \leq s \leq 1} A_T(g(s)) = A_T(q_2), \qquad g:[0,1] \rightarrow H^1_T(S^1), \qquad g \in G \]
where $G$ denotes the family of continuous maps of the interval [0,1] into $H^1_T(S^1)$ which satisfy $g(0) = q_1, g(1) = q_1 + 2\pi$. Recall that a mountain pass critical value  has the following topological property 
\[ \mathcal U^-(c) = \left \{ q \in H^1_T(S^1) | A_T(q) < c \right \} \] 
is not path connected. This condition means that the Morse index (and hence the Conley-Zehnder index as well) for a mountain pass $q_2$ must satisfy $i_1(q_2) \leq 1$ with equality when $q_2$ is nondegenerate in the sense that the nullity of the second variation is zero when considered on the space of T-periodic variations along $q_2$.  

  By transversality, each of the critical points $q_i$ corresponds to a T-periodic solution $x_i$ of the forced pendulum equation [\ref{fp}]. The following theorem from \cite{ort} represents the best known global result on elliptic stable periodic solutions for the forced pendulum equation. By global we mean results  which do not require a perturbation argument assuming that $f$ is small in some norm. Denote the set of continuous and T-periodic functions with mean value zero by 
\[ X = \left \{ f \in C\left ( \mathbb R / T \mathbb Z \right ) | \int_0^T f(t)dt = 0 \right \} . \]
\begin{theorem} (Ortega) Assume that the parameter $\beta$ satisfies the condition 
\begin{equation} \label{beta}  0 < \beta \leq \left ( \frac{\pi}{T} \right )^2 \end{equation}
Then the set $\mathcal S = \left \{ f \in X | (\ref{fp}) \text{ has a stable T-periodic solution} \right \} $ is prevalent in X. 
\end{theorem}
For an explanation of the term prevalent, we refer to Ortega. Our goal here is to recapture this result, using the ideas in the current paper. 

Given a T-periodic solution u(t) for the pendulum equation (\ref{fp}), we consider the variational equation
\[  \ddot x + \beta \cos(u(t)) x = 0. \] 
we will use the idea mentioned in the last section concerning the change in the argument of the solution $z(t) =(x(t),y(t)), y = -\dot x$, for the variational equation which is 
\[ \Delta \theta_T = 2 \int_0^T \frac{H(x(t),y(t), t, \beta)}{x^2 + y^2}dt, \qquad 2H(x,y,t,\beta) = y^2 + \beta \cos(u(t)) x^2  \] We will refer to a projective line in $P^1(\R^2)$ as a Lagrangian line. 
\begin{lemma} \label{compare}
Assuming that $f(t)$ is not the zero function, given the condition (\ref{beta}),   every Lagrangian line solution of the variational equation rotates by an argument less than  $k\pi$ over $k$ periods of time.
\end{lemma}
\begin{proof} We use a direct comparison argument for the change in the argument of the Lagrangian line, as the time parameter increases by the period $kT$. 
\[ \Delta \theta_{kT} =  \int_0^{kT} \frac{y^2 + \beta \cos(u(t)) x^2}{x^2 + y^2}dt  \leq   \int_0^{kT} \frac{y^2 + \beta  x^2}{x^2 + y^2}dt  = 
kT \sqrt{\beta}. \]
The reason for the equality in the last term of this calculation is that the integrand is precisely the rate of change of the argument for a harmonic oscillator of frequency $\sqrt{\beta}$, independent of $(x,y)$.  With the condition (\ref{beta}) we find that the argument of the Lagrangian line would increase by no more than $k\pi$. If this increase were equal to $k\pi$ then we would need to have $\cos(u(t))=1$ for all t. However this can happen for a periodic solution $u(t)$ of the pendulum equation (\ref{fp}) only when $f(t)$ is the zero function. Therefore the increase in the argument is less than $k\pi$. 
\end{proof} 
A comment about the case when $f(t) =0$. The proof of the Lemma indicates that the constant solutions of the pendulum equation, in the absence of forcing, will rotate by at most $\pi$  in the time interval [0,T], when the condition $ 0 < \beta \leq \left ( \frac{\pi}{T} \right )^2$ is met. This implies that the equilibrium solution at the origin is a mountain pass critical point of the functional $A_T(q)$ when $f=0$. However the double cover of the equilibrium, may not be a mountain pass in this case. 
 
As a result of this Lemma for $k=1$, we can observe immediately that if (\ref{beta}) holds and $f(t)$ is not identically zero, then a hyperbolic periodic orbit $u(t)$ must have real and positive Floquet multipliers. Thus we can conclude the mountain pass periodic solution $q_2$ is elliptic if it is non-degenerate. The reason is that its Conley-Zehnder index $i_1(q_2)=1$, thus the Floquet multipliers of $q_2$ can only be in $\{a<0: a\in\R\}\cup S^1\setminus\{1\}$. But the Floquet multipliers cannot be negative, hence $q_2$ must be elliptic. 

However we may pursue a different line of argument. Lemma \ref{compare} for $k=2$ tells us that the index for the double cover of $u(t)$ is no bigger than 1, i.e. $i_2(u)\leq 1$. Thus a non-degenerate mountain pass periodic solution $q_2$ has indices $i_1(q_2)=i_2(q_2)=1$, thus $q_2$ is elliptic by Theorem \ref{thm_dc}. 

In terms of non-degeneracy, we will use Ortega's result \cite{ort13} that there is an open and prevalent set $\mathcal R\subset X$ so that every $T$-periodic solution of (\ref{fp}) is non-degenerate for $f\in \mathcal R$. Our function $f(t)$ will always be in the set $\mathcal R\setminus\{0\}$.


\end{document}